\documentclass{amsart}

\title{More rigid ideals}
\author{Monroe Eskew}

\address{Kurt G\"odel Research Center \\
University of Vienna \\
25 W\"ahringer Strasse \\
1090 Wien, Austria.}

\date{}

\usepackage{color} 
\usepackage{amssymb} 
\usepackage{amsmath} 
\usepackage[utf8]{inputenc} 
\usepackage{enumerate}
\usepackage{amsthm}
\usepackage{cite}

\newtheorem{theorem}{Theorem}
\newtheorem{lemma}[theorem]{Lemma}
\newtheorem{proposition}[theorem]{Proposition}
\newtheorem{corollary}[theorem]{Corollary}

\newtheorem{claim}[theorem]{Claim}

\DeclareMathOperator{\dom}{dom}

\DeclareMathOperator{\ot}{ot}

\DeclareMathOperator{\cf}{cf}
\DeclareMathOperator{\cof}{cof}

\DeclareMathOperator{\col}{Col}
\DeclareMathOperator{\ord}{Ord}

\DeclareMathOperator{\supp}{supp}
\DeclareMathOperator{\crit}{crit}

\DeclareMathOperator{\ns}{NS}

\DeclareMathOperator{\spl}{Spl}
\DeclareMathOperator{\rcol}{\mathbb R \mathbb C}

\newcommand{\p}{\mathcal{P}}
\newcommand{\la}{\langle}
\newcommand{\ra}{\rangle}

\begin{document}
\maketitle

\begin{abstract}
We extend the results of \cite{rigid1}, showing the consistency of GCH with the statement that for all regular cardinals $\kappa \leq \lambda$, where $\kappa$ is the successor of a regular cardinal, there is a rigid saturated ideal on $\p_\kappa\lambda$.  We also show the consistency of some instances of rigid saturated  ideals on $\p_\kappa\lambda$ where $\kappa$ is the successor of a singular cardinal.
\end{abstract}

\section{Introduction}

A structure is said to be \emph{rigid} if it has no nontrivial automorphisms.  It follows from the principle Martin's Maximum (MM) that the boolean algebra $\p(\omega_1)/\ns$ is rigid.\footnote{See \cite{fms1}, \cite{larsonunique},  \cite{woodin-book}.}  An important component of this argument is that, under MM, this boolean algebra satisfies the $\omega_2$-c.c., or in other words $\ns_{\omega_1}$ is \emph{saturated}.  Saturated ideals are a way for small cardinals to mimic some properties of very large cardinals, by being the critical points of elementary embeddings between relatively rich transitive models that come about via relatively mild forcing.\footnote{See \cite{foremanhandbook} for background.}  The idea for rigidity in this context is to arrange that the forcing codes information into the manipulation of sufficiently absolute properties, which correlate to the details of the embedding, so that only one embedding can exist.

We will call an ideal $I$ over a set $X$ rigid when the quotient boolean algebra $\p(X)/I$ is rigid.  In \cite{rigid1}, Brent Cody and the author showed that it is possible to have rigid saturated ideals on other successor cardinals $\kappa = \mu^+$, by forcing an analogue of Martin's Axiom.  In these models, we have $2^\mu > \kappa$, and so we also investigated whether one can have rigid saturated ideals with GCH.  We were able to construct models of GCH with saturated rigid ideals on cardinals of the form $\kappa = \mu^+$, where $\mu$ is regular and uncountable, by having the ideal code information about manipulating the stationarity of subsets of $\mu$.  This technique is not possible for $\mu= \omega$, since the notion of stationarity trivializes there.  Here, we introduce a different coding method that solves the case $\mu = \omega$ and allows for a global result, which is not obviously achievable with the method of \cite{rigid1}.

\begin{theorem}
\label{global}
If ZFC is consistent with a huge cardinal, then there is a model of ZFC+GCH in which for every pair of regular cardinals $\kappa \leq \lambda$, where $\kappa$ is the successor of a regular cardinal $\mu$, there is a normal $\kappa$-complete ideal $I$ on $\p_\kappa\lambda$ such that $\p(\p_\kappa\lambda)/I$ is rigid, $\lambda^+$-c.c., and has a $\mu$-closed dense subset. 
\end{theorem}

Addressing the case of successors of singulars, we have:

\begin{theorem}
\label{sing}
Suppose $\theta$ is a huge cardinal and $\nu < \theta$ is regular and uncountable. Then there is a forcing extension in which for some $\theta' < \theta$, $V_{\theta'} \models$ ZFC + GCH + ``There is a singular cardinal $\mu$ of cofinality $\nu$ such that for every regular $\lambda \geq \kappa=\mu^+$ there is a normal $\kappa$-complete ideal $I$ on $\p_\kappa\lambda$ such that $\p(\p_\kappa\lambda)/I$ is rigid and $\lambda^+$-c.c.''
\end{theorem}

In Section \ref{rigcolsec}, we introduce a certain restricted product of L\'evy collapses which is rigid, preserves GCH, and changes a Mahlo cardinal into the successor of a chosen regular cardinal.  We prove Theorem \ref{global} in Section \ref{idealsec}.  In Section \ref{singsec}, we prove Theorem \ref{sing} and discuss why it is hard to combine the conclusions of the two theorems into one model.

The full generality of the following lemma is only applied in Section \ref{singsec}, but we state it here because its interest extends beyond the present topics.
\begin{lemma}
 \label{absoluteness}
 Suppose $\kappa$ is a regular uncountable cardinal, $\mathbb P$ is $\kappa$-c.c., $\mathbb Q$ is $\kappa$-strategically closed, and $\dot{\frak A}$ is a $\mathbb P$-name for a structure of size $\kappa$ in a language $\dot{\mathcal L}$.  Suppose $G \subseteq \mathbb P$ is generic and that in $V[G]$, $S$ is a set of symbols, $T$ is a set of sentences in $\mathcal L' = \mathcal L \cup S$, and $|T|<\kappa$.  If $\mathbb Q$ forces over $V[G]$ that there is an expansion $\frak A'$ of $\frak A$ to the language $\mathcal L'$ such that $\frak A' \models T$, then this is already true in $V[G]$.
\end{lemma}

We note the following corollaries.  Suppose $\kappa, \mathbb P, \mathbb Q$ are as above.
\begin{enumerate}
\item In $V^{\mathbb P}$, $\mathbb Q$ cannot change the $\Sigma^1_1$ theory of structures of size $\leq \kappa$ with a language of size $<\kappa$.
 \item In $V^{\mathbb P}$, $\mathbb Q$ preserves $\kappa$-c.c.\ partial orders.
  \item In $V^{\mathbb P}$, $\mathbb Q$ preserves stationary subsets of $\kappa$.
 \item If $T$ is a tree of size $\leq \kappa$ in $V^{\mathbb P}$ and it has a cofinal branch in $V^{\mathbb P \times \mathbb Q}$, then it has a cofinal branch in $V^{\mathbb P}$.
 \item If $\frak A$ is a rigid structure of size $\leq \kappa$ in a language of size $<\kappa$ in $V^{\mathbb P}$, then $\frak A$ is rigid in $V^{\mathbb P \times \mathbb Q}$.
\end{enumerate}

\begin{proof}[Proof of Lemma \ref{absoluteness}]
We may assume that the domain of $\frak A$ is forced to be $\kappa$.
 In $V^{\mathbb P}$ let $\mathcal L''$ be the subset of $\mathcal L'$ that is mentioned in $T$; we have $|\mathcal L''| <\kappa$.  In $V^{\mathbb P \times \mathbb Q}$, let $\frak A'' = \frak A' \restriction \mathcal L''$.  In this extension, there is a club $C \subseteq \kappa$ such that $\frak A'' \restriction \alpha \prec \frak A''$ for all $\alpha \in C$.
 
 Let $\la \dot s_i : i < \delta < \kappa \ra$ be a sequence of $\mathbb P$-names for the elements of $\dot S \cap \dot{\mathcal L}''$, and for each $i<\delta$, let $\dot X_i$ be a $\mathbb P \times \mathbb Q$-name for the interpretation of $s_i$ in $\frak A'$.  Let $\dot X$ be a $\mathbb P \times \mathbb Q$-name for a subset of $\kappa$ which canonically codes $\la X_i : i < \delta \ra$.
  
  Let $\dot C$ be a $\mathbb P \times \mathbb Q$-name for the club $C$ above.   Let $\sigma$ witness the strategic closure of $\mathbb Q$.
 Let $(p_0^0,q^0_0)$ force some ordinal $\alpha_0^0$ to be in $\dot C$, and decide whether $0 \in \dot X$.  Suppose we have a $\la (p_0^i,q_0^i),\alpha_0^i : i <j \ra$ such that the $p_0^i$'s form an antichain, the $q_0^i$'s are a descending sequence following $\sigma$, the $\alpha_0^i$'s are increasing, and $(p_0^i,q_0^i)$ forces $\alpha_0^i \in \dot C$ and decides whether $0 \in \dot X$.  If the $p_0^i$'s do not form a maximal antichain, pick $p^j_0,q^j_0,\alpha_0^j$ such that for all $i<j$, $p_0^j \perp p_0^i$ and $\alpha_0^j > \alpha_0^i$, $q_0^j$ is a lower bound to the $q_0^i$'s following $\sigma$,  and $(p_0^j,q_0^j) \Vdash \alpha_0^j \in \dot C$.  If the antichain is maximal, let $A_0 = \{ p_0^i : i < j \}$, let $q_0$ be a lower bound to the chain of $q_0^i$'s chosen according to $\sigma$, and let $\alpha_0 = \sup_{i<j} (\alpha_0^j+1)$.  By the chain condition, this will occur at some $j<\kappa$.

 Continue in this way, getting a sequence of maximal antichains in $\mathbb P$, $\la A_i : i < \kappa \ra$, a descending sequence in $\mathbb Q$ that follows $\sigma$, $\la q_i : i < \kappa \ra$, and an increasing continuous sequence of ordinals $\la \alpha_i : i < \kappa \ra$, such that for all $i$ and all $p \in A_i$, $(p,q_{i})$ decides whether $i \in \dot X$, and if $i$ is a limit, then $(1,q_i) \Vdash \alpha_i \in \dot C$.
 
 Let $G \subseteq \mathbb P$ be generic.  For each $i < \kappa$, there is a unique $p_i \in A_i \cap G$.  Let 
 $$Y = \{ i : (p_i,q_i) \Vdash^V_{\mathbb P \times \mathbb Q} i \in \dot X \} = \{ i : (\exists p \in G)(\exists j<\kappa) (p,q_j) \Vdash^V_{\mathbb P \times \mathbb Q} i \in \dot X \}.$$
 We claim that interpreting the symbols $s_i$ according to $Y$ produces an expansion $\frak B$ of $\frak A$ that satisfies $T$. 
 There is a club of limit ordinals $D \subseteq \kappa$ such that for all $\beta \in D$, $\alpha_\beta = \beta$ and $\beta$ is closed under our coding scheme, in the sense that $Y \cap \beta$ codes a sequence $\la Y_i^\beta : i < \delta \ra$ of interpretations for $\la s_i \restriction \beta : i <\delta \ra$.  For such $\beta$, we have that $(1,q_\beta) \Vdash \frak A'' \restriction \beta \prec \frak A''$, and thus that $\frak A'' \restriction \beta \models T$.  Since $(p_i,q_\beta)$ decides whether $i \in \dot X$ for all $i < \beta$, $q_\beta \Vdash^{V[G]}_{\mathbb Q} \dot X \cap \beta = Y \cap \beta$, and thus $q_\beta \Vdash^{V[G]}_{\mathbb Q} \frak B \restriction \beta = \frak A'' \restriction \beta$.  Since $\la \frak B \restriction \beta : \beta \in D \ra$ forms an elementary chain, $\frak B \models T$.
 \end{proof}

\section{A rigid collapse}
\label{rigcolsec}

This section is devoted to a proof of the following:
\begin{theorem}
\label{main}
 Suppose $\kappa$ is a Mahlo cardinal and $\mu < \kappa$ is regular.
 Then there is a $\mu$-closed, $\kappa$-c.c.\ partial order $\rcol \subseteq V_\kappa$
forcing $\kappa = \mu^+$, and whenever $G \subseteq \mathbb \rcol$ is generic over $V$,
then in $V[G]$, $G$ is the unique filter which is $\rcol$-generic over $V$.
\end{theorem}

Let us start with a few well-known forcing facts.  If $\mathbb P$ and $\mathbb Q$ are partial orders with maximum elements $1_{\mathbb P}$ and $1_{\mathbb Q}$ respectively, a map $\pi : \mathbb P \to \mathbb Q$ is called a \emph{projection} if $\pi(1_{\mathbb P}) = 1_{\mathbb Q}$, $\pi$ is order-preserving, and $\pi$ has the property that whenever $q \leq \pi(p)$, there is $p' \leq p$ such that $\pi(p') \leq q$.  A map $e : \mathbb P \to \mathbb Q$ is called an \emph{embedding} if it is order- and antichain-preserving.  An embedding is \emph{regular} if it preserves maximal antichains and \emph{dense} if its range is a dense subset of the codomain.  For every partial order $\mathbb P$, there is a complete boolean algebra $\mathcal B(\mathbb P)$ and a dense embedding $d : \mathbb P \to \mathcal B(\mathbb P)$.

\begin{lemma}
 Suppose $\pi : \mathbb P \to \mathbb Q$ is a projection.  If $H \subseteq \mathbb Q$, we write $\mathbb P/H$ for $\pi^{-1}[H]$.
 \begin{enumerate}
  \item If $G$ is $\mathbb P$-generic over $V$ and $H = \pi[G]$, then $H$ is $\mathbb Q$-generic over $V$, and $G$ is $\mathbb P/H$-generic over $V[H]$.
  \item If $H$ is $\mathbb Q$-generic over $V$, and $G$ is $\mathbb P/H$-generic over $V[H]$, then $G$ is $\mathbb P$-generic over $V$.
 \end{enumerate}
\end{lemma}

\begin{lemma}
 If $d : \mathbb P \to \mathbb Q$ is a dense embedding, then $G$ is $\mathbb P$-generic over $V$ iff $d[G]$ is $\mathbb Q$-generic over $V$.
\end{lemma}

For a set of ordinals $X$ and an ordinal $\alpha$, we use $[X]^\alpha$ to denote the collection of subsets of $X$ of ordertype $\alpha$.
Suppose $V \subseteq W$ are models of set theory.  We define the following 
$\Sigma_1$ statement about parameters in $V$:
$$\spl(\mu,\kappa,V):  \left( \exists A \in [\kappa]^\kappa \right) (\forall x \in [\kappa]^\mu \cap V)
(\forall \alpha < \mu)
(\exists y \in [x]^\alpha \cap V) 
 y \cap A = \emptyset.$$
Informally, this says that there is a large subset of $\kappa$ in $W$ which 
splits every $\mu$-sized set from $V$ by excluding arbitrarily large pieces of it from $V$.
In cases where the inner model in question is clear from context,
we will drop the third parameter and just write $\spl(\mu,\kappa)$.
Clearly, if $V$ and $W$ have the same $[\kappa]^\mu$, then $W \models \neg \spl(\mu,\kappa,V)$.

\begin{lemma}
\label{pos}
Suppose $\mu<\kappa$ are regular.
Then $\col(\mu,{<}\kappa)$ forces $\spl(\mu,\kappa)$.
\end{lemma}

\begin{proof}
Let $G \subseteq \col(\mu,{<}\kappa)$ be generic.  In $V[G]$, define $A = \{ 
\alpha < \kappa : \bigcup G(\alpha,0)= 0 \}$.  Let $p \in 
\col(\mu,{<}\kappa)$, $\alpha < \mu$, and $x \in [\kappa]^\mu \cap V$ be 
arbitrary.  Since $|p| < \mu$, there is $y \subseteq x 
\setminus \supp(p)$ of ordertype $\alpha$. We can construct $q \leq p$ such 
that for all $\beta \in y$, $q(\beta,0) \not= 0$.  Therefore, the set of conditions forcing witnesses to 
$\spl(\mu,\kappa)$ is dense, so the desired statement is forced.
\end{proof}

\begin{lemma}
\label{negs}
Suppose $\nu < \mu < \kappa$ are regular and $\alpha^{<\nu} < \kappa$ for all 
$\alpha < \kappa$.  Then:
\begin{enumerate}
\item $\Vdash_{\col(\nu,{<}\kappa)} \neg \spl(\mu,\kappa)$.
\item $\Vdash_{\col(\mu,{<}\kappa)} \neg \spl(\nu,\kappa)$.
\end{enumerate}
\end{lemma}

\begin{proof}
(2) holds since $\col(\mu,{<}\kappa)$ does not change $[\kappa]^\nu$.
For (1), let $\dot A$ be a $\col(\nu,{<}\kappa)$-name for a set in $[\kappa]^\kappa$. 
 Suppose $q \in \col(\nu,{<}\kappa)$ forces that for all $x \in [\kappa]^\mu \cap V$, 
there is $y \in [x]^\nu \cap V$
that is disjoint from $\dot A$.  Let $B$ be the set of $\alpha < \kappa$ such that there is some $p_\alpha 
\leq q$ such that $p_\alpha \Vdash \check \alpha \in \dot A$.  Let $C \in 
[\kappa]^{\kappa}$ be such that $\{ p_\alpha : \alpha \in C \}$ forms a 
$\Delta$-system with root $r \leq q$. Let $x \in [C]^\mu$, and suppose $s \leq 
r$ decides some $y \in [x]^\nu$ to be disjoint from 
$\dot A$.
Since $|\supp(s)| < \nu$, 
there is some $\alpha \in y$ such that $p_\alpha$ is compatible with $s$, so $s$ 
does not force that $\check y \cap \dot A = \emptyset$, a contradiction.  Thus, 
$\neg \spl(\mu,\kappa)$ is forced.
\end{proof}

Recall that a set of ordinals $X$ is Easton when for all regular cardinals $\kappa$, \break$\sup(\kappa \cap X )< \kappa$.  Below, a superscript $E$ above a product will indicate that we take all partial functions with Easton support.

\begin{lemma}
\label{holes}
 Suppose $\kappa$ is Mahlo.  Let $X \subseteq \kappa$ be a set of regular cardinals such that for some regular $\mu < \kappa$,
 $\mu^+ \notin X$.
 Then the partial order
 $$\mathbb P = \prod^E_{\alpha \in X} \col(\alpha,{<}\kappa)$$
 is $\kappa$-c.c.\ and forces $\neg \spl(\mu^+,\kappa)$.
\end{lemma}

\begin{proof}
We establish the $\kappa$-c.c.\ using $\Delta$-systems.
Let $\{ p_\alpha : \alpha < \kappa \} \subseteq \mathbb P$.
For each $\alpha < \kappa$, there is $\beta < \kappa$ such that
$$p_\alpha \in \mathbb P \restriction \beta := \prod^E_{\gamma \in X \cap \beta} \col(\gamma,{<}\beta).$$
Let $\beta_\alpha$ denote the least such $\beta$, and 
let $C \subseteq \kappa$ be a club closed under the function $\alpha \mapsto \beta_\alpha$.
Note that for all $\alpha < \kappa$ and for all regular $\beta < \kappa$, there is $\gamma < \beta$ such that
$p_\alpha \restriction \beta \in \mathbb P \restriction \gamma$.
Since $\kappa$ is Mahlo, there is a $\gamma^* < \kappa$ and a stationary $S \subseteq C$ such that
for all $\alpha \in S$, $p_\alpha \restriction \alpha \in \mathbb P \restriction \gamma^*$.
Since $|\mathbb P \restriction \gamma^*| < \kappa$, there is a stationary $S' \subseteq S$ 
and $p^* \in \mathbb P \restriction \gamma^*$ such that $p_\alpha \restriction \alpha = p^*$ for all
$\alpha \in S'$.  If $\alpha < \beta$ are in $S'$, then $p_\beta \restriction \dom p_\alpha = p^*$,
so $p_\alpha$ and $p_\beta$ are compatible.

Now note that
$$\mathbb P \cong \prod^E_{\alpha \in [0,\mu] \cap X} \col(\alpha,{<}\kappa)
\times \prod^E_{\alpha \in [\mu^{++},\kappa) \cap X} \col(\alpha,{<}\kappa)
:= \mathbb P_0 \times \mathbb P_1,$$
and $\mathbb P_1$ is $\mu^{++}$-closed.
Let $G_1 \subseteq \mathbb P_1$ be generic and work in $V[G_1]$.  Suppose $q \Vdash^{V[G_1]}_{\mathbb P_0} \dot A \in [\kappa]^\kappa$.  When possible, let $p_\alpha \leq q$ be such that $p_\alpha \Vdash \alpha \in \dot A$.
Let $\la \alpha_i : i < \kappa \ra$ enumerate the set of $\alpha$ for which $p_\alpha$ is defined.
Since each condition in $\mathbb P_0$ has size $<\mu$, 
we can find a stationary $S \subseteq \kappa \cap \cof(\mu)$ such that $\{ p_{\alpha_i} : i \in S \}$ forms a $\Delta$-system with root $r \leq q$ as above. 
As in the proof of Lemma \ref{negs}, for every $y \in [S]^\mu$ and every $s \leq r$,
$s \nVdash \{\alpha_i : i \in y \} \cap \dot A = \emptyset$, since $|s| <\mu$.
This shows that $\mathbb P_0$ forces $\neg \spl(\mu^+,\kappa,V[G_1])$ over $V[G_1]$.
Since $([\kappa]^{\mu^+})^V = ([\kappa]^{\mu^+})^{V[G_1]}$, $\mathbb P$ forces $\neg \spl(\mu^+,\kappa,V)$.
\end{proof}

Suppose $\kappa$ is Mahlo and $\mu < \kappa$ is regular.  We define the rigid collapse $\rcol$ of $\kappa$ to $\mu^+$ as a projection of a product of L\'evy collapses.  Let $\mathrm{Reg}$ denote the class of regular cardinals.
Let $$\mathbb P = \prod^E_{\alpha \in [\mu,\kappa) \cap \mathrm{Reg}} \col(\alpha,{<}\kappa).$$
$\mathbb P$ can be viewed as the set of partial functions $p : \kappa^3 \to \kappa$ such that:
\begin{enumerate}
\item $\{ \alpha : (\exists \beta)( \exists \gamma) (\alpha,\beta,\gamma) \in \dom p \}$ is an Easton set of regular cardinals contained in $[\mu,\kappa)$.
\item $(\forall \alpha) |\{ \beta : (\exists \gamma) (\alpha,\beta,\gamma) \in \dom p \} | < \alpha$.
\item $(\forall \alpha)( \forall \beta) \{ \gamma : (\alpha,\beta,\gamma) \in \dom p \} \in [\alpha]^{<\alpha}$.
\item $(\forall \alpha)( \forall \beta )(\forall \gamma) p(\alpha,\beta,\gamma) <\beta$.
\end{enumerate}
Enumerate quadruples of ordinals by putting $\vec \alpha < \vec \beta$ when $\max \vec \alpha < \max \vec \beta$, or if not, $\vec \alpha$ is lexicographically less than $\vec \beta$.  For every infinite cardinal $\alpha$, this enumeration has ordertype $\alpha$.  Let $f : \kappa \to \kappa^4$ be the restriction of this enumeration to $\kappa$.
A generic filter $G$ for $\mathbb P$, or any suborder, is determined by the collection
$$\{ (\alpha,\beta,\gamma,\delta) : (\exists p \in G)p(\alpha,\beta,\gamma) = \delta \},$$
and thus by a subset of $\kappa$ via $f$.

We want to define $\rcol$ so that it absorbs other versions itself where we alter the choice of the cardinal $\mu$ but keep $\kappa$ the same.  In order to arrange this, we divide the regular cardinals below $\kappa$ into countably many pieces as follows.  
Let $A_0$ be the set:
\begin{align*}
  \{ \alpha \in [\mu,\kappa) :& \text{ } \alpha = \mu, \text{ or }\alpha\text{ is inaccessible, or }
  \alpha = \beta^{+n} 
  \text{ for some singular cardinal }\beta \\
 &\text{ of cofinality}>\mu\text{ and some finite }n > 0 \} \times \kappa \times \kappa.
\end{align*}
For $n > 0$, let $A_n$ be the set:
$$\{ \alpha \in [\mu,\kappa) : \alpha = \beta^{+n+1}, \text{ for some singular cardinal } \beta \text{ of cofinality } \mu \} \times \kappa \times \kappa.$$
We will inductively define iterations $\mathbb P_0 * \dots * \mathbb P_n$ which are the images of commuting projections from the respective $\mathbb P \restriction \bigcup_{m\leq n} A_m$.  Let $\mathbb P_0 = \mathbb P \restriction A_0$.

Let $\la \alpha_i : i < \kappa \ra$ enumerate the singular cardinals of cofinality $\mu$ in $(\mu,\kappa)$ in increasing order.
Suppose $G_0 \subseteq \mathbb P_0$ is generic over $V$.  Let $X_0$ be the subset of $\kappa$ that codes $G_0$ via $f$ as above.  Let 
$B_1 = \{  \alpha_i^{++} : i \in X_0 \} \times \kappa^2 \subseteq A_1$, and let $\mathbb P_1 = \mathbb P \restriction B_1$.

\begin{claim}
\label{proj}
In $V[G_0]$, the map $p \mapsto p \restriction B_1$ is a projection from $\mathbb P$ to $\mathbb P_1$.\
\end{claim}

\begin{proof}
 Suppose $q \leq p \restriction B_1$.  Then $q = p' \restriction B_1$ for some $p'$ such that $p'(\alpha,\beta,\gamma) = p(\alpha,\beta,\gamma)$ whenever $(\alpha,\beta,\gamma) \in B_1 \cap \dom p$.
 Let $d = \dom p' \setminus \dom p$, and define $p'' = p \cup (p' \restriction d)$.
 Then $p'' \leq p$, and $p'' \restriction B_1 = q$.
\end{proof}

\begin{claim}
 If $i \notin X_0$, then $\Vdash^{V[G_0]}_{\mathbb P_1} \neg \spl(\alpha_i^{++},\kappa,V)$.
\end{claim}

\begin{proof}
If $i \notin X_0$, then the same argument for the previous claim shows that in $V[G_0]$, there is a projection
 $\pi : \mathbb P \restriction A_1 \setminus (\{ \alpha_i^{++} \} \times \kappa^2) \to \mathbb P_1.$
Lemma \ref{holes} implies that any generic extension by $\mathbb P \restriction (\kappa \setminus \{ \alpha_i^{++} \}) \times \kappa^2$ satisfies $\neg \spl(\alpha_i^{++},\kappa,V)$.
 If $G_1$ is $\mathbb P_1$-generic over $V[G_0]$, then a further forcing gives a generic $H \subseteq \mathbb P \restriction (\kappa \setminus 
 \{\alpha_i^{++} \}) \times \kappa^2$ over $V$.
 Since $\spl(\alpha_i^{++},\kappa,V)$ is a $\Sigma_1$ property of parameters from $V$, $\neg \spl(\alpha_i^{++},\kappa,V)$ holds in $V[G_0][G_1]$.
\end{proof}

\begin{claim}
 Whenever $G_0 * G_1$ is $\mathbb P_0 * \mathbb P_1$-generic over $V$, and $G_0' * G_1' \in V[G_0 * G_1]$ is also 
 $\mathbb P_0 * \mathbb P_1$-generic over $V$, then $G_0 = G_0'$.
\end{claim}

\begin{proof}
 Suppose otherwise.  Let $X_0$ and $X_0'$ be the subsets of $\kappa$ corresponding to $G_0$ and $G_0'$ respectively.
 There must be some ordered quadruple $(\alpha,\beta,\gamma,\delta) \in G_0' \setminus G_0$, and thus some
 $i \in X_0' \setminus X_0$.  By the definition of $\mathbb P_1$ and Lemma \ref{pos}, $V[G_0' * G_1'] \models
 \spl(\alpha_i^{++},\kappa,V)$.  But by the previous claim, $V[G_0 * G_1] \models \neg \spl(\alpha_i^{++},\kappa,V)$.
 This is impossible, as $\spl(\alpha_i^{++},\kappa,V)$ is $\Sigma_1$ in parameters from $V$, and $V[G_0' * G_1'] \subseteq V[G_0 * G_1]$.
\end{proof}

Now we simply continue this process $\omega$ many times.  Suppose that we have sequences 
$\la \mathbb P_j : j < n \ra$,
$\la X_j : j < n \ra$, and
$\la B_j : 1 \leq j \leq n \ra$,
such that for $m < n$,
\begin{enumerate}
 \item $X_m$ is a $(\mathbb P_0 * \dots *\mathbb P_{m})$-name for the subset of $\kappa$ which codes the generic $G_m$
 for $\mathbb P_m$ via $f$, and $B_{m+1}$ is  a name for $\{ \alpha_i^{+m+2} : i \in X_m \} \times \kappa^2 \subseteq A_{m+1}$.
 \item It is forced by $\mathbb P_0 * \dots *\mathbb P_{m-1}$ that $\mathbb P_m = \mathbb P \restriction B_m$, and $p \mapsto p \restriction B_m$ is a projection.
\end{enumerate}
We extend these properties to sequences of length $n+1$ using the same argument as in Claim \ref{proj}.

Now we define $\rcol$ as a limit of this sequence.  The elements of $\rcol$ are just the elements of $\mathbb P \restriction \bigcup_{n<\omega} A_n$, but their ordering is different.  We put $p \leq_{\rcol} q$ when for each $n$,
$$\la p \restriction A_0,\check p \restriction \dot B_1,\dots,\check p \restriction \dot B_n \ra \leq_{\mathbb P_0 * \dots * \mathbb P_n} \la q \restriction A_0,\check q \restriction \dot  B_1,\dots,\check q \restriction\dot B_n \ra.$$
The ordering extends the superset ordering on $\mathbb P$.
Note that this only defines a preorder since we may have distinct conditions $p,q \in \mathbb P \restriction \bigcup_{n<\omega} A_n$ such that for all $n$, 
$\la p \restriction A_0,\dots,\check p \restriction \dot B_n \ra \Vdash_{\mathbb P_0 * \dots * \mathbb P_{n}} \check p \restriction \dot B_{n+1}= \check q \restriction \dot B_{n+1}$.  As usual, we may take the quotient by the equivalence relation defined by $p \sim q$ when $p \leq q \leq p$.  Modulo this equivalence relation, we have for each $n$, $\rcol \restriction \bigcup_{m\leq n} A_m$ is isomorphic to a dense subset of $\mathbb P_0 * \dots * \mathbb P_n$.

We want to show that the identity map from $\mathbb P \restriction \bigcup_{n<\omega} A_n$ to $\rcol$ is a projection.  Suppose $q \leq_{\rcol} p$.  
As before, let $d = \dom q \setminus \dom p$, 
and define $p' = p \cup (q \restriction d)$ so that $p' \leq_{\mathbb P} p$.  
We must have $p' \restriction A_0 = q \restriction A_0 \supseteq p \restriction A_0$.
It follows by induction that for each $n>0$, 
$\la q \restriction A_0,\dots,\check q \restriction \dot B_n \ra \Vdash_{\mathbb P_0 * \dots * \mathbb P_{n}} \check p' \restriction \dot B_{n+1} = \check q \restriction \dot B_{n+1} \supseteq \check p \restriction \dot B_{n+1},$
and thus $p'$ is equivalent to $q$ in $\rcol$.

To show rigidity, suppose $G \subseteq \rcol$ is generic over $V$, and $G' \in V[G]$ is also $\rcol$-generic over $V$.
Let $n$ be least such that $G \restriction A_n \not= G' \restriction A_n$.
As above, there is $i \in X'_{n} \setminus X_{n}$.
Suppose $p \in G$ forces $i \notin X_{n}$.  Then there is a projection
$$\sigma : \mathbb P \restriction (\kappa \setminus \{ \alpha_i^{+n+2} \}) \times \kappa^2 \to \rcol \restriction p.$$
A further forcing yields a filter $H \subseteq \mathbb P \restriction (\kappa \setminus \{ \alpha_i^{+n+2} \}) \times \kappa^2$,
which is generic over $V$, and such that $V[G] \subseteq V[H]$.  By Lemma \ref{holes}, 
$V[H] \models \neg \spl(\alpha_i^{+n+2},\kappa,V)$, and so must $V[G']$.  But by the construction and Lemma \ref{pos},
$V[G'] \models \spl(\alpha_i^{+n+2},\kappa,V)$, a contradiction.  

To complete the proof of Theorem \ref{main}, we only need to show $\mu$-closure.

\begin{lemma}
\label{closure}
Any directed subset of $\rcol$ of size $<\mu$ has an infimum.
\end{lemma}
\begin{proof}
  Suppose $\nu < \mu$ and $\la p_i : i < \nu \ra$ is such a set.
Let $d$ be the set of $(\alpha,\beta,\gamma) \in \bigcup_{i<\nu} \dom p_i$ such that there are no $i < j < \nu$ with $p_i(\alpha,\beta,\gamma) \not= p_j(\alpha,\beta,\gamma)$.  Let $p^* = \bigcup_{i<\nu} p_i \restriction d$.  We show by induction on $n<\omega$ that $p^* \restriction \bigcup_{m\leq n} A_m = \inf_{i<\nu}(p_i \restriction \bigcup_{m \leq n} A_m)$ in $\rcol$.  For $n = 0$, this is true since the functions $p_i \restriction A_0$ for all $i<\nu$ all agree on the points in common to their domains.

Suppose this is true for $n$.  Let $G \subseteq \mathbb P_0 * \dots * \mathbb P_{n}$ be generic with $\la p^* \restriction A_0,\dots,\check p^* \restriction \dot B_n \ra \in G$.  By induction, $\la p_i \restriction A_0,\dots,\check p_i \restriction \dot B_n \ra \in G$ for all $i< \nu$, so by directedness, $\{ p_i \restriction B_{n+1} : i < \nu \}$ is a set of pairwise compatible partial functions on $B_{n+1}$.
No two $p_i,p_j$ can disagree at a point in $B_{n+1}$,
so $B_{n+1} \subseteq d$.  Therefore $\bigcup_{i<\nu} p_i \restriction B_{n+1} = p^* \restriction B_{n+1}$, and this is evidently the greatest lower bound of the $p_i \restriction B_{n+1}$ for $i<\nu$.  As $G$ was arbitrary, the desired statement is forced at $n+1$.
\end{proof}

If $G \subseteq \rcol$ is generic, then $\mathbb P/G$ is equivalent to $\mathbb P \restriction  \kappa^3 \setminus ( A_0 \cup  \bigcup_{n \geq 1} B_n)$.
Therefore we have:
\begin{proposition}
 If $G \subseteq \rcol$ is generic, then $\mathbb P/G$ is $\mu$-directed-closed.
\end{proposition}

The rigidty of $\rcol$ is relatively robust:
\begin{lemma}
\label{robust}
 Suppose $|\mathbb Q| < \kappa$.  If $G \subseteq \rcol$ is generic and $H \subseteq \mathbb Q$ is generic over $V[G]$, then in $V[G][H]$, $G$ is still the only filter which is $\rcol$-generic over $V$.
\end{lemma}
\begin{proof}
 Suppose the contrary.  If $G' \in V[G][H]$ were another $\rcol$-generic filter, then there would be some regular $\nu < \kappa$ such that $V[G'] \models \spl(\nu,\kappa,V)$ and $V[G] \models \neg \spl(\nu,\kappa,V)$.  Let $\dot A$ be a $\mathbb P$-name in $V[G]$ for a set forced to witness $\spl(\nu,\kappa,V)$.  Since $|\mathbb Q|<\kappa$, there is $B \in V[G]$ of size $\kappa$ and $p \in H$ such that $p \Vdash \check B \subseteq \dot A$.  However, $B$ also witnesses $\spl(\nu,\kappa,V)$, since for any $x \in [\kappa]^\nu \cap V$ and any $\alpha < \nu$, there is $y \in [x]^\alpha \cap V$ such that $\dot A^H \cap y = \emptyset$, and thus $B \cap y = \emptyset$.  Contradiction.
\end{proof}

The construction above of $\mathbb P$, $A_n$, $X_n$, $B_n$, $\mathbb P_n$, and $\rcol$ were relative to parameters $\mu$ and $\kappa$, so let us indicate this by writing $\mathbb P(\mu,\kappa)$, $\rcol(\mu,\kappa)$, etc., and let us write $\mathbb Q(\mu,\kappa)$ for $\mathbb P(\mu,\kappa) \restriction \bigcup_{n<\omega} A_n(\mu,\kappa)$.
We would like to record some useful facts about projections between different rigid collapses:
\begin{lemma}
\label{projs}
 Suppose $\mu < \kappa \leq \lambda < \delta$ are regular and $\lambda$ is inaccessible.
 \begin{enumerate}
\item There is a projection 
$\sigma : \rcol(\mu,\delta) \to \rcol(\mu,\kappa) \times \mathbb Q(\lambda,\delta)$
\item If $\pi : \rcol(\mu,\delta) \to \rcol(\mu,\kappa)$ is given by the first coordinate of the output of $\sigma$, and $X \subseteq  \rcol(\mu,\delta)$ is a directed set of size $<\mu$, then $\pi(\inf X) = \inf \pi[X]$.
\end{enumerate}
\end{lemma}
\begin{proof}
Note that for each $n$ and all regular cardinals $\alpha < \beta < \gamma$, $A_n(\alpha,\gamma) \cap V_\beta = A_n(\alpha,\beta)$, so we write $A_n^\alpha$ for the class $\bigcup_{\beta \in \ord} A_n(\alpha,\beta)$.
The map $\sigma$ is defined by $\sigma(p) = (p \restriction \kappa^3, p \restriction \bigcup_{n<\omega} A^\lambda_n)$.   The second coordinate works because $\bigcup_{n<\omega} A^\lambda_n \subseteq A^\mu_0$.  (The fact that $\lambda$ is inaccessible ensures that $A^\lambda_0 \subseteq A^\mu_0$.)

For the first coordinate, the key point is that the enumeration of quadruples of ordinals was canonical.  
We show by induction on restrictions to the sets $A^\mu_n$ that 
\begin{enumerate}[(a)]
 \item $p \leq_{\rcol(\mu,\delta)} q$ implies $p \restriction \kappa^3 \leq_{\rcol(\mu,\kappa)} q \restriction \kappa^3$, and
 \item $p \leq_{\rcol(\mu,\kappa)} q \restriction \kappa^3$ implies $p \cup q \restriction (\delta^3 \setminus \kappa^3) \leq_{\rcol(\mu,\delta)} q$.
\end{enumerate}
The base case is clear since the ordering restricted to $A^\mu_0$ is just the superset relation.  Suppose this is true for partial functions whose domains are contained in $\bigcup_{m< n} A^\mu_m$.
Suppose first $p,q \in \mathbb P(\mu,\delta) \restriction \bigcup_{m \leq n} A^\mu_m$, and $p \leq_{\rcol(\mu,\delta)} q$.  We must show $p \restriction \kappa^3 \leq_{\rcol(\mu,\kappa)} q \restriction \kappa^3$.  If $p \supseteq q$, then we are done, so the only situations to worry about are the points $(\alpha,\beta,\gamma) \in \dom q \cap A^\mu_n \cap \kappa^3$ such that $p(\alpha,\beta,\gamma) \not= q(\alpha,\beta,\gamma)$.  For such a point, it must be the case that some value of $p \restriction \bigcup_{m<n}A^\mu_{m}$ forces some ordinal $j$ to be not in $X_{n-1}$, where $\alpha = \alpha_j^{+n+1}$.  This is determined by the enumeration of quadruples of ordinals, and the increasing enumeration of the singular cardinals of cofinality $\mu$, both of which have the property that the enumeration up the rank $\kappa$ is an initial segment of the one up to rank $\delta$.  Thus, it is information in $p \cap V_\kappa$ that decides $\alpha$ is excluded from $X_{n-1}$, and the same data decide the same result in $\rcol(\mu,\kappa)$.  We conclude that $p \restriction \kappa^3 \leq_{\rcol(\mu,\kappa)} q \restriction \kappa^3$, and (a) follows by induction.  The same argument shows (b), and together these imply that $p \mapsto p \restriction \kappa^3$ is a projection from $\rcol(\mu,\delta)$ to $\rcol(\mu,\kappa)$.  

To show (2), let $X \subseteq \rcol(\mu,\delta)$ be a directed set of size $<\mu$.  The proof of Lemma \ref{closure} shows that $\inf X$ is given by $\bigcup X \setminus B$, where $B = \{ (\alpha,\beta,\gamma,\delta) : (\exists p,q \in X) p(\alpha,\beta,\gamma) \not= q(\alpha,\beta,\gamma) \}$.  $\pi[X]$ is also directed, and thus has an infimum in $\rcol(\mu,\kappa)$ defined by the same operation, $\bigcup \pi[X] \setminus B$.  Thus,
$\pi(\inf X) = (\bigcup X \setminus B) \restriction \kappa^3 = \bigcup_{p \in X} p \restriction \kappa^3 \setminus B = \inf \pi[X].$
\end{proof}

Suppose $\mathbb P$ is a partial order and $\dot{\mathbb Q}$ is a partial order in $V^{\mathbb P}$.  The \emph{termspace forcing}, $T(\mathbb P,\dot{\mathbb Q})$ is the collection of $\mathbb P$-names for elements of $\mathbb Q$ (in $H_\theta$, where $\theta$ is regular and $\mathbb P,\dot{\mathbb Q} \in H_\theta$), ordered by $\dot q_1 \leq \dot q_0$ iff $1 \Vdash_{\mathbb P} \dot q_1 \leq \dot q_0$.  It is easy to see that if $\dot{\mathbb Q}$ is forced to be $\kappa$-closed, then $T(\mathbb P,\dot{\mathbb Q})$ is $\kappa$-closed.  This idea is due to Laver, and we show now a slight generalization of the main lemma which Laver proved about this notion.

\begin{lemma}
\label{termspace}
 Suppose $\pi : \mathbb P \to \mathbb R$ is a projection, and $\dot{\mathbb Q}$ is an $\mathbb R$-name for a partial order.  Then $\dot{\mathbb Q}$ can be interpreted as a $\mathbb P$-name, and the identity map is a projection from $\mathbb P \times T(\mathbb R,\dot{\mathbb Q})$ to $\mathbb P * \dot{\mathbb Q}$.
\end{lemma}

\begin{proof}
Suppose $(p_0,\dot q_0) \in \mathbb P \times T(\mathbb R,\dot{\mathbb Q})$, and $(p_1,\dot q_1) \leq_{\mathbb P * \dot{\mathbb Q}} (p_0,\dot q_0)$.  Find $p_2 \leq p_1$ and an $\mathbb R$-name $\dot q_2$ such that $p_2 \Vdash \dot q_1 = \dot q_2$.  Then build an $\mathbb R$-name $\dot q_3$ such that $\pi(p_2) \Vdash \dot q_3 = \dot q_2$ and $r \Vdash \dot q_3 = \dot q_0$ whenever $r \perp \pi(p_2)$.  Then $(p_2,\dot q_3) \leq_{\mathbb P * \dot{\mathbb Q}} (p_1,\dot q_1)$, since whenever $H$ is generic for $\mathbb P$ with $p_2 \in H$, then $\dot q_3^{\pi[H]} = \dot q_2^{\pi[H]} = \dot q_1^H$.   Furthermore, $1 \Vdash_{\mathbb R} \dot q_3 \leq \dot q_0$.  For let $G$ be generic for $\mathbb R$.  If $\pi(p_2) \notin G$, then $\dot q_3^G = \dot q_0^G$.  If $\pi(p_2) \in G$, then we can do further forcing to produce $H$ generic for $\mathbb P$, with $G = \pi[H]$.  Since $p_2 \Vdash \dot q_3 = \dot q_2 = \dot q_1 \leq \dot q_0$, it must already be true in $V[G]$ that $\dot q_3^G \leq \dot q_0^G$.
\end{proof}

Shioya \cite{shioyanew} showed that if $\kappa$ is regular and $\mathbb P$ is a $\kappa$-c.c.\ partial order of size $\leq \kappa$, then for every $\delta$ such that $\delta^{<\kappa} = \delta$, there is a dense embedding $d : \col(\kappa,{<}\delta) \to T(\mathbb P,\dot{\col}(\kappa,{<}\delta))$.  The same argument shows the following for our Easton products of L\'evy collapses:

\begin{lemma}
\label{shioya}
Assume GCH, $\kappa$ is regular, and $\mathbb R$ is a $\kappa$-c.c.\ partial order of size $\leq \kappa$.  Then for every regular $\delta \geq \kappa$, there is a dense embedding $d : \mathbb Q(\kappa,\delta) \to T(\mathbb R,\dot{\mathbb Q}(\kappa,\delta))$ 
\end{lemma}

We also need the following folklore result:\footnote{See \cite{cummingshandbook}.}

\begin{lemma}
\label{folk}
If $\kappa$ is a regular cardinal and $\mathbb P$ is a $\kappa$-closed partial order forcing $|\mathbb P| = \kappa$, then there is a dense embedding $d : \col(\kappa,|\mathbb P|) \to \mathbb P$.
\end{lemma}

\begin{corollary}
\label{proj2}
 Suppose $\mu < \kappa \leq \lambda < \delta$ are regular, $\kappa$ is Mahlo, and $\lambda$ is inaccessible.  Suppose $|\mathbb P| < \kappa$ and $\sigma$ is a $\mathbb P \times \rcol(\mu,\kappa)$-name for a projection from $\col(\kappa,\lambda)$ to some partial order $\dot{\mathbb R}$, which is forced to be $\lambda$-c.c.
 Then there is a projection 
 $$ \pi : \mathbb P \times \rcol(\mu,\delta) \to (\mathbb P \times \rcol(\mu,\kappa)) * \dot{\col}(\kappa,\lambda)  * \dot{\rcol}(\lambda,\delta)^{V^{(\mathbb P \times \rcol(\mu,\kappa)) * \dot{\mathbb R}}}.$$
\end{corollary}

\begin{proof}
 Using Lemma \ref{projs}, there is a projection
 $$\pi_0 : \mathbb P \times \rcol(\mu,\delta) \to \mathbb P \times \rcol(\mu,\kappa) \times \col(\mu,\lambda) \times \mathbb Q(\lambda,\delta).$$
 Lemma \ref{folk} implies that there is a dense embedding
 $$d_0 : \col(\mu,\lambda) \to \col(\mu,\lambda) \times T(\mathbb P \times \rcol(\mu,\kappa),\dot{\col}(\kappa,\lambda)).$$
Therefore, using Lemma \ref{termspace}, there is a projection $\pi_1$ from the codomain of $\pi_0$ to
$$((\mathbb P \times \rcol(\mu,\kappa))  * \dot{\col}(\kappa,\lambda)) \times \mathbb Q(\lambda,\delta).$$
 By Lemma \ref{shioya}, there is also a dense embedding 
 $$d_1: \mathbb Q(\lambda,\delta) \to T((\mathbb P \times \rcol(\mu,\kappa)) * \dot{\mathbb R},\dot{\mathbb Q}(\lambda,\delta)).$$
 Combining these gives us a projeciton $\pi_2$ from the codmain of $\pi_1$ to
 $$(\mathbb P \times \rcol(\mu,\kappa)) * \dot{\col}(\kappa,\lambda)) * (\dot{\mathbb Q}(\lambda,\delta))^{V^{(\mathbb P \times \rcol(\mu,\kappa)) * \dot{\mathbb R}}}.$$
Finally, in $V^{(\mathbb P \times \rcol(\mu,\kappa)) * \dot{\mathbb R}}$, there is a projection $\pi_3 : \mathbb Q(\lambda,\delta) \to \rcol(\lambda,\delta)$.   Applying this to the last term above yields the desired projection.
 \end{proof}
 
 A close examination of the construction of the above projection reveals:
 \begin{proposition}
 \label{closedquotient}
 Suppose that in the hypotheses of the previous corollary, $\mathbb P$ is $\nu$-distributive, for some $\nu \leq \mu$.  If $G$ is generic for the image of the projection, then the quotient forcing $\rcol(\mu,\delta)/G$ is $\nu$-directed-closed.
 \end{proposition}

\section{Saturated rigid ideals}
\label{idealsec}
Let us recall some basic facts about saturated ideals, proof of which can be found in \cite{foremanhandbook}.  An ideal $I$ on $Z \subseteq \p(\lambda)$ is \emph{normal} when for all sequences $\la A_\alpha : \alpha < \lambda \ra \subseteq I$, the diagonal union, $\nabla_{\alpha<\lambda} := \{ z : (\exists \alpha) \alpha \in z \in A_\alpha \}$, is in $I$ as well.  Least upper bounds in the boolean algebra $\p(Z)/I$ are given by diagonal unions.  Therefore, if $\p(Z)/I$ has the $\lambda^+$-chain condition---synonymously, $I$ is saturated---then $\p(Z)/I$ is a complete boolean algebra.  Whenever $G \subseteq \p(Z)/I$ is generic, then the generic ultrapower $V^Z/G$ is well-founded and closed under $\lambda$-sequences from $V[G]$.  ($I$ is called \emph{precipitous} it always yields well-founded generic ultrapowers.)

A cardinal $\kappa$ is called \emph{huge} if it is the critial point of an elementary embedding $j : V \to M$, where $M$ is a transitive class such that $M^{j(\kappa)} \subseteq M$.  A cardinal is called \emph{almost-huge} when we only require $M^{<j(\kappa)} \subseteq M$.  We will need the following facts about almost-huge embeddings, which can be found in \cite{kanamori}:

\begin{lemma}
\label{ahmin}
Suppose $\kappa$ is almost-huge, witnessed by an embedding sending $\kappa$ to $\delta$.  Then there is an elementary $j : V \to M$ with the following properties:
\begin{enumerate}
\item The embedding is generated by a tower of measures $T \subseteq V_\delta$, which we will call a \emph{$(\kappa,\delta)$-tower}.
The fact that $T$ generates such an embedding is equivalent to a first-order property of $( V_\delta, \in, T )$.
\item $\crit j = \kappa$, $j(\kappa) = \delta$, and $M^{<\delta} \subseteq M$.
\item $\sup j[\delta] = j(\delta) < \delta^+$.
\end{enumerate}
\end{lemma}

The following is proven by standard reflection arguments:

\begin{proposition}
\label{tourney}
 If $\kappa$ is huge, then there is an unbounded set $A \subseteq \kappa$ such that for every $\alpha < \beta$ in $A$, there is an $(\alpha,\beta)$-tower.
\end{proposition}

 \begin{lemma}
 \label{idealfromah}
 Suppose $\mu < \kappa \leq \lambda < \delta$, $\mathbb P$ and $\dot{\mathbb R}$ are as in the hypothesis of Lemma \ref{proj2}.  Suppose additionally there is a $(\kappa,\delta)$-tower and $(\mathbb P \times \rcol(\mu,\kappa)) * \dot{\mathbb R}$ preserves the regularity of some $\gamma \in [\kappa,\lambda]$.  Then there is a projection
 $$\pi : \mathbb P \times \rcol(\mu,\delta) \to (\mathbb P \times \rcol(\mu,\kappa)) * \dot{\mathbb R} * (\dot{\col}(\gamma,\lambda) \times \dot{\rcol}(\lambda,\delta)),$$
and whenever $G$ is generic for the righthand side, then in $V[G]$ there is a normal $\kappa$-complete ideal $I$ on $\p_\kappa\gamma$ such that $\p(\p_\kappa\gamma)/I \cong (\mathbb P \times \rcol(\mu,\delta))/G.$
\end{lemma}

\begin{proof}
In $V^{\mathbb P \times \rcol(\mu,\kappa)}$, there is a dense embedding $d : \col(\kappa,\lambda) \to \col(\kappa,\lambda) \times (\mathbb R * \dot{\col}(\gamma,\lambda))$.  Combining this with the projection of Lemma \ref{proj2} gives the desired projection.
Suppose $$ G * h_0 * (h_1 \times H) \subseteq (\mathbb P \times \rcol(\mu,\kappa)) * \dot{\mathbb R} * (\dot{\col}(\gamma,\lambda) \times \dot{\rcol}(\lambda,\delta))$$ is generic over $V$.
We may force further to produce $\hat G$ that is $\mathbb P \times \rcol(\mu,\delta)$-generic and projects to $G * h_0 * (h_1 \times H)$.  If $j : V \to M$ is an almost-huge embedding generated by a $(\kappa,\delta)$-tower, then we may extend the embedding to $j : V[G] \to M[\hat G]$.  By the $\delta$-c.c.\ of $\rcol(\mu,\delta)$, $M[\hat G]^{<\delta} \cap V[\hat G] \subseteq M[\hat G]$.
 
$\hat G$ absorbs a generic $h \subseteq \col(\kappa,\lambda)^{V[G]}$ that projects to $h_0 * h_1$.  Since $\col(\delta,j(\lambda))^{M[\hat G]}$ is $\delta$-directed closed, and $j[h]$ is a directed subset of size $<\delta$, we can take a lower bound $q^* \in \col(\delta,j(\lambda))^{M[\hat G]}$.  Since $M[\hat G] \models 2^{j(\lambda)} < j(\delta)$, and $j(\delta) < (\delta^+)^V$, we can build a generic $\hat h$ for $\col(\delta,j(\lambda))^{M[\hat G]}$ in $V[\hat G]$ below $q^*$.  This projects to an $M[\hat G]$-generic $\hat h_0 * \hat h_1 \subseteq j(\mathbb R) * \dot{\col}(j(\gamma),j(\lambda))$, and thus we may extend the embedding again to $j : V[G * h_0 * h_1] \to M[\hat G * \hat h_0 * \hat h_1]$.
 
 For each regular $\alpha \in (\lambda,\delta)$, let $m_\alpha = \inf j[H {\restriction} \alpha^3]$.  The proof of Lemma \ref{closure} shows that $m_\alpha$ is given by $\bigcup \{ j(p) : p \in H \restriction \alpha \} \setminus B$, where $B$ is the set of coordinates where there is some disagreement between two of the partial functions.  Thus we may assume $m_\alpha  = m_\beta \restriction j(\alpha)^3$ for $\alpha < \beta$.  Since $M[\hat G * \hat h_0] \models \rcol(j(\lambda),j(\delta))$ is $j(\delta)$-c.c., and $j(\delta)<(\delta^+)^V$, we can enumerate all maximal antichains of this partial order in $V[\hat G]$ as $\la A_\alpha : \alpha <\delta \ra$.  For each $A_\alpha$, there is a regular $\beta_\alpha < \delta$ such that $A_\alpha \subseteq \rcol(j(\lambda),j(\beta_\alpha))^{M[\hat G * \hat h_0]}$.  We may assume the $\beta_\alpha$'s are increasing.  We can inductively build a descending chain $\la q_\alpha : \alpha < \delta \ra \subseteq \rcol(j(\lambda),j(\delta))^{M[\hat G * \hat h_0]}$ such that: 
 \begin{enumerate}
  \item For each $\alpha$, $q_\alpha \in \rcol(j(\lambda),j(\beta_\alpha))^{M[\hat G * \hat h_0]}$.
  \item For each $\alpha$, $q_\alpha \supseteq m_{\beta_\alpha}$, and $q_\alpha \leq r$ for some $r \in A_\alpha$.
  \item For $\alpha < \alpha'$, $q_\alpha \subseteq q_{\alpha'}$. 
   \end{enumerate}
  Suppose the construction has proceeded up to $\xi$.  For each $i < \xi$, $m_{\beta_\xi} \restriction \beta_i^3 = m_{\beta_i} \subseteq q_i$, and thus $q'_\xi := \bigcup_{i<\xi} q_i \cup m_{\beta_\xi}$ is a condition.  Find some $r \in A_\alpha$ compatible with $q'_\xi$, and let $q_\xi \supseteq q'_\xi$ be stronger than $r$.  In the end, this generates a filter $\hat H$ which is generic over $M[\hat G * \hat h_0 * \hat h_1]$ and includes $j[H]$.  Thus we can extend the embedding to $j : V[G*h_0 * (h_1 \times H)] \to M[\hat G*\hat h_0*(\hat h_1 \times \hat H)]$.
  
 In $V[G*h_0 * (h_1 \times H)]$, we define a normal ideal $I$ by $X \in I$ iff $\Vdash j[\gamma] \notin j(X)$. The map $e : [X]_I \mapsto || j[\gamma] \in j(X) ||$ is easily seen to be an embedding of $\p(\p_\kappa\gamma)/I$ into $\mathcal B(\frac{\mathbb P \times \rcol(\mu,\delta)}{G*h_0*(h_1\times H)})$.  Since the latter is $\delta$-c.c. and $\delta=\gamma^+$, $I$ is saturated.  Let $\la A_\alpha : \alpha < \gamma \ra$ be a maximal antichain in $\p(\p_\kappa\gamma)/I$.  Then it is forced that $ j[\gamma] \in j(\nabla_{\alpha<\gamma} A_\alpha)$, so by the definition of diagonal unions it is forced that for some $\alpha < \gamma$, $j[\gamma] \in j(A_\alpha)$.  Thus $e$ is a regular embedding.
 
 If $U \subseteq \p(\p_\kappa\gamma)/I$ is generic over $V[G*h_0*(h_1\times H)]$, then there is an elementary embedding $i : V[G*h_0*(h_1\times H)] \to N$, where $N$ is transitive.  By further forcing, we obtain an extension of the ground-model embedding $j$ as above, and we have an elementary embedding $k : N \to M[\hat G*\hat h_0 *(\hat h_1 \times \hat H)]$ defined by $k([f]_U) = j(f)(j[\gamma])$.  The definition of $e$ guarantees that $k$ is elementary, and clearly $j(x) = k(i(x))$ for all $x \in V[G*h_0*(h_1\times H)]$.
 
 Since $\mu$ is fixed by $j$ and $\delta = (\mu^+)^N = i(\kappa) = j(\kappa)$, we have $\crit(k) > \delta.$  Therefore, whenever $U \subseteq \p(\p_\kappa\gamma/I)$ is generic and we proceed to produce $j$ and $k$, we have $i(G) = j(G) = \hat G$ as above, so no further forcing is needed to get a generic for $\frac{\mathbb P \times \rcol(\mu,\delta)}{G*h_0*(h_1\times H)}$.  Furthermore, for every condition $p$ in the quotient forcing, a generic $U$ can be taken yielding a generic $\hat G$ with $p \in \hat G$.  If $[X]_I$ forces $p \in \hat G$, then $e([X]_I) \leq p$.  Thus $e$ is a dense embedding, and the conclusion follows.
 \end{proof}

We are now ready for:
\begin{proof}[Proof of Theorem \ref{global}]
Suppose $\theta$ is huge, and let $A \subseteq \theta$ be as in Proposition \ref{tourney}.  Let $\kappa$ be the first inaccessible limit point of $A$, and let $\la \alpha_i : i < \kappa \ra$ be the increasing enumeration of the closure of $A \cap \kappa \cup \{\omega \}$.

We define the following Easton-support iteration $\mathbb P = \la \mathbb P_i,\dot{\mathbb Q}_i : i < \kappa \ra$.  As usual, $\mathbb P_0$ is the trivial partial order.  Let $\mathbb P_1 = \mathbb Q_0 = \rcol(\omega,\alpha_1)$.   If $i$ is a successor ordinal, let $\mathbb P_{i+1} = \mathbb P_i * \dot{\mathbb Q}_i = \mathbb P_i * \dot{\rcol}(\alpha_i,\alpha_{i+1})$.

If $i<\kappa$ is a limit ordinal, then $\alpha_i$ is singular.  Let $\lambda_i$ be the first inaccessible in $(\alpha_i,\alpha_{i+1})$.  Let $\dot{\mathbb Q_i}$ be a $\mathbb P_i$-name for $\col(\alpha_i^+,\lambda_i) \times \rcol(\lambda_i,\alpha_{i+1})$. 

The following are easy to see:
\begin{enumerate}
\item If $i$ is a successor, $\mathbb P_i$ is $\alpha_i$-c.c., and $\mathbb P / \mathbb P_i$ is $\alpha_i$-closed.
\item If $i$ is a limit, $\mathbb P_i$ is $\lambda_i$-c.c., and $\mathbb P / \mathbb P_i$ is $\alpha_i^+$-closed.
\item $\mathbb P$ preserves the inaccessbility of $\kappa$.
\item $\mathbb P$ forces that the set of infinite cardinals below $\kappa$ is $\{ \alpha_i : i < \kappa \} \cup \{ \alpha_i^+ : i < \kappa$ is a limit$\}$.
\end{enumerate}

Let $i < j < \kappa$ be successor ordinals.  
Since $|\mathbb P_{i-1}|<\alpha_i$, it preserves the existence of an $(\alpha_i,\alpha_j)$-tower.  In $V^{\mathbb P_{i-1}}$, the forcing $\mathbb P_j / \mathbb P_{i-1}$ takes the form:
$$(\mathbb Q \times \rcol(\alpha_{i-1},\alpha_i)) * \dot{\mathbb R} * (\dot{\col}(\gamma,\lambda) \times \dot{\rcol}(\lambda,\alpha_j)),$$
where $|\mathbb Q| < \alpha_i$, $\dot{\mathbb R}$ is forced to be $\alpha_i$-closed and $\lambda$-c.c., $\gamma \leq \lambda$ is forced to be regular, and $\lambda < \alpha_j$ is inaccessible in the ground model.  (If $i-1$ is a limit, $\mathbb Q = \col(\alpha_{i-1}^+,\lambda_{i-1})$, and $\mathbb Q$ is trivial otherwise.  If $j-1$ is a limit, then $\gamma = \alpha_{j-1}^+$ and $\lambda = \lambda_{j-1}$; otherwise $\gamma = \lambda = \alpha_{j-1}$.)  Thus Lemma \ref{idealfromah} applies, and $\mathbb P_j$ forces that there is a normal $\alpha_i$-complete ideal $I$ on $\p_{\alpha_i}\gamma$ with quotient algebra equivalent to $(\mathbb Q \times \rcol(\alpha_{i-1},\alpha_j))/G$, where $G$ is the generic for $\mathbb P_j/ \mathbb P_{i-1}$.

Suppose that, in $V^{\mathbb P_j}$, there is a nontrivial automorphism of $\p(\p_{\alpha_i}\gamma)/I$.  Let $V' = V^{\mathbb P_{i-1}}$.  Forcing with $\p(\p_{\alpha_i}\gamma)/I$ and applying the automorphism would produce a $\mathbb Q \times \rcol(\alpha_{i-1},\alpha_j)$-extension of $V'$ with two distinct generics, $g \times G$ and $g' \times G'$.  Since $\mathbb Q$ is contained in the projection of the forcing to $\mathbb P_j/\mathbb P_{i-1}$, we have $g = g'$.  Lemma \ref{robust} implies that $V'[g][G]$ has only one filter which is $\rcol(\alpha_{i-1},\alpha_j)^{V'}$-generic over $V'[g]$, so $G = G'$, contradicting the assumption.

Lemma \ref{absoluteness} implies that the $\alpha_j$-closed quotient forcing $\mathbb P / \mathbb P_j$ preserves that the $\alpha_j$-sized boolean algebra $\p(\p_{\alpha_i}\gamma)/I$ is $\alpha_j$-c.c.\ and rigid.
Finally, Proposition \ref{closedquotient} justifies the claim that $\p(\p_{\alpha_i}\gamma)/I$ is equivalent to a $\nu$-closed forcing, where $\nu = \alpha_{i-1}^+$ if $i$ is a limit, and $\nu = \alpha_{i-1}$ otherwise.  Cutting the universe at $\kappa$ produces a model of Theorem \ref{global}.
\end{proof}

We would have a bit of an easier time if we did not concern ourselves with saturated ideals on $\p_\kappa\lambda$, for $\lambda$ successor of singular.  We could just let $\mathbb Q_i$ be $\rcol(\alpha_i^+,\alpha_{i+1})$ at limit $i$, and use a simpler version of Lemma \ref{projs}.

\section{Successors of singulars}
\label{singsec}

It is a well-known theorem of Laver \cite{laverindestructible} that the supercompactness of a cardinal $\kappa$ can be made indestructible under $\kappa$-directed-closed forcing.  An examination of his proof reveals the following more specific result, which we will use:

\begin{theorem}[Laver]
 Suppose $\kappa$ is supercompact.  There is an iteration $\mathbb I = \la \mathbb I_i, \dot{\mathbb J}_i : i < \kappa \ra \subseteq V_\kappa$ with the following property:  Whenver $G \subseteq \mathbb I$ is generic over $V$, $\mathbb Q$ is a $\kappa$-directed-closed forcing in $V[G]$ of size $\leq \lambda$, and $H \subseteq \mathbb Q$ is generic over $V[G]$, then in $V[G*H]$, there is a normal $\kappa$-complete ultrafilter $U$ on $\p_\kappa\lambda$ such that $j_U(G) \restriction (\kappa +1) = G *H$, and $j_U(\mathbb I)/(G*H)$ is $\lambda^+$-closed.
 \end{theorem}
 
 Start in a model $V'$ in which $\theta$ is huge and $\nu < \theta$ is regular and uncountable.  Let $A \subseteq \theta$ be as in Proposition \ref{tourney}.  Let $\mu>\nu$ be such that $V_\theta \models \mu$ is supercompact, and let $\eta$ be the first inaccessible limit point of $A$ above $\mu$.  Let $\la \alpha_i : i < \eta \ra$ be the increasing enumeration of the closure of $\{ \mu \} \cup (A \setminus \mu) \cap \eta$.  First force with Laver's partial order over $V'$ to obtain a model $V=V'[G]$.  This preserves all the $(\alpha,\beta)$-towers for $\alpha<\beta$ from $A \setminus (\mu+1)$.
 
 Next, we define an Easton-support iteration, $\mathbb P = \la \mathbb P_i,\dot{\mathbb Q}_i : i < \eta \ra$.  Let $\mathbb Q_0 = \rcol(\mu,\alpha_1)$.   If $i$ is a successor ordinal, let $\dot{\mathbb Q}_i$ be a $\mathbb P_i$-name for $\rcol(\alpha_i,\alpha_{i+1})$.  If $i$ is a limit ordinal, let $\lambda_i$ be the first inaccessible in $(\alpha_i,\alpha_{i+1})$, and let $\dot{\mathbb Q_i}$ be a $\mathbb P_i$-name for $\col(\alpha_i^+,\lambda_i) \times \rcol(\lambda_i,\alpha_{i+1})$.
 
 Suppose $\lambda$ is forced to be a regular cardinal in $(\mu,\eta)$ after this iteration, let $i$ be least such that $\alpha_i > \lambda$.  Put $\kappa = \alpha_1$ and $\delta = \alpha_i$.  It suffices to prove that there is a $\mathbb P_i$-name for an $\delta$-c.c.\ forcing $\dot{\mathbb Q}$ such that $\mathbb P_i * \dot{\mathbb Q}$ forces $\cf(\mu) = \nu$, and $\mathbb P_i * \dot{\mathbb Q}$ forces that there is a normal, $\kappa$-complete, rigid, saturated ideal on $\p_\kappa\lambda$.  For then, since $\mathbb Q$ is $\delta$-c.c.\ and $\mathbb P/\mathbb P_i$ is $\delta$-closed in $\mathbb P_i$, $\mathbb P/\mathbb P_i$ adds no subsets of $\lambda$ over $V^{\mathbb P_i * \dot{\mathbb Q}}$ by Easton's Theorem, and Lemma \ref{absoluteness} shows that the saturation and rigidity of the ideal is preserved by  $\mathbb P/\mathbb P_i$.

By Lemma \ref{idealfromah}, $\mathbb P_i$ forces that there is a normal $\kappa$-complete ideal $I_\lambda$ on $\p_{\kappa}\lambda$ such that $\p(\p_{\kappa}\lambda)/I_\lambda \cong \mathcal B(\rcol(\mu,\delta)/\mathbb P_i)$.  We use the following to analyze what happens to our ideal after further forcing:

\begin{theorem}[Foreman \cite{foremanduality}]
\label{dualitynicecase}
If $I$ is a $\kappa$-complete precipitous normal ideal on $Z \subseteq \p(\lambda)$ and $\mathbb{P}$ is $\kappa$-c.c.  Let $\bar I$ denote the ideal generated by $I$ in $V^{\mathbb P}$, and let $j$ denote the generic ultrapower embedding associated to forcing with $\p(Z)/I$.  There is an isomorphism $$\iota : \mathcal B(\mathbb{P} * \dot\p(Z)/ \bar{I}) \cong \mathcal B(\p(Z)/I * \dot{j(\mathbb{P})})$$ defined by $\iota(p,\dot X) = (1,\dot{j(p)}) \wedge || j[\lambda] \in j(\dot X) ||$.
\end{theorem}

Let $H_i$ be $\mathbb P_i$-generic over $V$, and let $H = H_i \restriction \rcol(\mu,\kappa)$.  By Laver's Theorem, let $U$ be a $\mu$-complete normal ultrafilter on $\p_\mu\kappa$ in $V[H]$ such that $j_U(G) \restriction (\mu + 1) = G * H$.  We force with the Radin forcing $\mathbb Q_\nu$ derived from $U$ in $V[H]$ to change the cofinality of $\mu$ to $\nu$.  Let us describe some of the important details of this forcing, which can be found in \cite{gitikhandbook}:
\begin{enumerate}
 \item $\mathbb Q_\nu$ can be written as $\bigcup_{i<\mu} X_i$, where each $X_i$ consists of pairwise compatible elements.  Thus $\mathbb Q_\nu$ is $\kappa$-c.c.
 \item $\mathbb Q_\nu$ preserves that $\mu$ is a limit cardinal and that $\nu$ is regular.
 \item $\mathbb Q_\nu$ is definable from the ultrafilter $U$.  More specifically, $\mathbb Q_\nu$ is constructed from a sequence $\vec u$ of $\mu$-complete measures on $V_\mu$, of length $\nu$, each derived from the ultrapower embedding $j_U$.  The first nontrivial measure of $\vec u$ is $\vec u(1) = \{ X \subseteq V_\mu : \mu \in j_U(X) \}.$
 \item $\mathbb Q_\nu$ adds a club $C \subseteq \mu$ of ordertype $\nu$, with the property that whenever $X \in \vec u(1)$, there is some $\beta < \mu$ such that $C \setminus \beta \subseteq X$.
 \item The generic filter for $\mathbb Q_\nu$ can be recovered from $C$.
\end{enumerate}

Although the forcing $\mathbb P_i/H$ adds subsets of $\kappa$ and thus makes $U$ no longer an ultrafilter on $\p_\mu\kappa$, $\mathbb P_i/H$ adds no subsets of $\mu$ nor $\mathbb Q_\nu$-names for subsets of $\mu$.
By item (1) above, Theorem \ref{dualitynicecase} implies that forcing with $\mathbb Q_\nu$ over $V[H_i]$ preserves the saturation of $I_\lambda$.  For it suffices to show that if $G^*$ is generic for $\p(\p_\kappa\lambda)/I_\lambda$ and $j : V[H_i] \to M \subseteq V[H_i][G^*]$ is the generic ultrapower embedding, then $j(\mathbb Q_\nu)$ is $j(\kappa)$-c.c.\ in $V[H_i][G^*]$.  Since the property of being the union of $\mu$ many sets of pairwise-compatible elements is upwards-absolute, this follows.

Let us argue that rigidity is preserved.
Let $e : \mathbb Q_\nu \to \mathcal B(\p(\p_{\kappa}\lambda)/I_\lambda * \dot{j(\mathbb Q_\nu)})$ be the restriction of the isomorphism $\iota$ to $\mathbb Q_\nu$.  Let $K \subseteq \mathbb Q_\nu$ be generic.  We have that  $\p(\p_{\kappa}\lambda)/\bar I_\lambda \cong B(\p(\p_{\kappa}\lambda)/I_\lambda * \dot{j(\mathbb Q_\nu)})/e[K]$.  If there were a nontrivial automorphism of $\p(\p_{\kappa}\lambda)/\bar I_\lambda$ in $V[H_i][K]$, then we would have a forcing extension of $V$ by $\rcol(\mu,\delta) * \dot{j(\mathbb Q_\nu)}$ with two distinct generics, $\hat H_0 * \hat K_0$ and $\hat H_1 * \hat K_1$, such that  $V[\hat H_0 * \hat K_0]= V[\hat H_1 * \hat K_1]$.

Now, since $\mu$ is below the critical point of the generic embedding $j$, the club $C \subseteq \mu$ associated to $K$ is the same as those associated to both $\hat K_0$ and $\hat K_1$, by the way the isomorphism $\iota$ is defined.  $\hat H_0$ and $\hat H_1$ determine generic ultrapowers $j_0,j_1$ of $V$ associated to the ideal $I_\lambda$.  If $\hat H_0 = \hat H_1$, then $j_0(\mathbb Q_\nu) = j_1(\mathbb Q_\nu)$, and thus $\hat K_0 = \hat K_1$ since these are definable from $C$ and the forcing $j_0(\mathbb Q_\nu)$.  So we must have $\hat H_0 \not= \hat H_1$.

Let $\beta < \delta$ be such that $\hat H_0$ includes a generic for $\col(\beta,{<}\delta)$ and $\hat H_1$ does not, so that $V[\hat H_0] \models \spl(\beta,\delta,V)$, and $V[\hat H_1] \models \neg \spl(\beta,\delta,V)$.  Both $j_0(\mathbb Q_\nu)$ and $j_1(\mathbb Q_\nu)$ are determined by normal ultrafilters $\hat U_0,\hat U_1$ on $\p_\mu\delta$ that live in $V[\hat H_0],V[\hat H_1]$ respectively.  Recall that $V = V'[G]$, where $G \subseteq V_\mu$ is generic for Laver's partial order.  If $i_0 : V[\hat H_0] \to N_0$ and $i_1 : V[\hat H_1] \to N_1$ are the respective ultrapower embeddings, then by Laver's Theorem, we have that $i_0(G) \restriction (\mu + 1) = \hat H_0$, and $i_1(G) \restriction (\mu + 1) = \hat H_1$.  Let $N_0 = N_0'[i_0(G * \hat H_0)]$ and $N_1 = N_1'[i_1(G * \hat H_1)]$.  Since the tail-ends of the respective iterations are sufficiently closed, we have that $N_0 \models \spl(\beta,\delta,N_0'[G])$ and $N_1 \models \neg \spl(\beta,\delta,N_1'[G])$.

Let $s_0,s_1$ be surjections from $\mu$ to $\beta$ such that $s_0 \in V[\hat H_0]$ and $s_1 \in V[\hat H_1]$.  We have that $\beta = \ot(i_0(s_0[\mu])) = \ot(i_1(s_1[\mu]))$.  Therefore, if $\vec u_0$ and $\vec u_1$ are the measure sequences associated to $j_0(\mathbb Q_\nu),j_1(\mathbb Q_\nu)$ respectively, then have that 
$$\{ \alpha < \mu : V[\hat H_0] \models \spl(\ot(s_0[\alpha]),\alpha^+,V'[G\restriction \alpha]) \} \in \vec u_0(1),$$
and
$$\{ \alpha < \mu : V[\hat H_1] \models \neg \spl(\ot(s_1[\alpha]),\alpha^+,V'[G\restriction \alpha]) \} \in \vec u_1(1).$$
Now, both $s_0$ and $s_1$ are in $V[\hat H_0 * \hat K_0]$.  Since $\mu$ has uncountable cofinality in this model, there is a club $D \subseteq \mu$ such that for all $\alpha \in D$, $s_0[\alpha] = s_1[\alpha]$. Therefore, there is a point $\alpha \in D \cap C$ and a $\gamma < \alpha^+$ such that $V[\hat H_0] \models \spl(\gamma,\alpha^+,V'[G\restriction \alpha])$ and $V[\hat H_1] \models \neg \spl(\gamma,\alpha^+,V'[G\restriction \alpha])$.  But since $V[\hat H_0]$ and $V[\hat H_1]$ have the same $\p(\alpha^+)$, this is a contradiction.  This completes the proof of Theorem \ref{sing}.

\hspace{1mm}

The conclusion of Theorem \ref{global} does not hold in the model of Theorem \ref{sing} constructed above.  This is because it is at odds with the use of Radin forcing.  Recall that the weak square principle at $\kappa$, abbreviated $\square^*_\kappa$, asserts that there is a sequence $\la \mathcal C_\alpha : \alpha < \kappa^+ \ra$ such that:
\begin{enumerate}
\item Each $\mathcal C_\alpha$ is a collection of club subsets of $\alpha$, and $|\mathcal C_\alpha | \leq  \kappa$.
\item If $C \in \mathcal C_\alpha$, and $\beta \in \lim C$, then $C \cap \beta \in \mathcal C_\beta$.
\end{enumerate}
If $\mu$ is inaccessible in some inner model with the same $\mu^+$, as in the proof of Theorem \ref{sing}, then it is easy to show that $\square^*_\mu$ holds.  The following proposition shows that the conclusion of Theorem \ref{global} implies that $\square^*_\mu$ fails for every singular $\mu$.  The argument is essentially the same as that for Theorem 10.1 in \cite{cfmstruc2}, but with slightly different hypotheses, and has nothing to do with rigidity.
\begin{proposition}
\label{wsquare}
Suppose $\kappa = \mu^+$, $\mu$ is regular, and $\lambda > \kappa$ is a strong limit cardinal of cofinality $<\mu$.  If there is a $\kappa$-complete normal ideal $I$ on $\p_\kappa\lambda'$, where $\lambda' > \lambda$, such that $\p(\p_\kappa\lambda')/I$ is $\mu$-strategically-closed, then $\square^*_\lambda$ fails.
\end{proposition}
\begin{proof}
Assume to the contrary that there exists such an ideal and a $\square^*_\lambda$-sequence $\vec{\mathcal C} = \la \mathcal C_\alpha : \alpha < \lambda^+ \ra$.  Since $\lambda$ is a strong limit, we can assume that if $C \in \mathcal C_\alpha$, and $D \subseteq C$ is club in $\alpha$ and has size $<\lambda$, then $D \in \mathcal C_\alpha$.

Let $G \subseteq \p(\p_\kappa\lambda')/I$ be generic, and let $j : V \to M \subseteq V[G]$ be the associated embedding.  Let $\gamma = (\lambda^+)^V$.  Since $j(\kappa) > \lambda'$, $M \models |\gamma| = \mu$, and by the strategic closure, $M \models \cf(\gamma) = \mu$.  Let $\delta = \sup j[\gamma] < j(\gamma)$, and let $C \in j(\vec{\mathcal C})_\delta$ have ordertype $\mu$.  Since $\mu < \crit(j)$, and $V$ and $V[G]$ share the same ${<}\mu$-sequences, $j[\gamma]$ is ${<}\mu$-closed.  Thus $D = C \cap j[\gamma]$ is club in $\delta$ and thus a member of $j(\vec{\mathcal C})_\delta$.  If $\alpha < \gamma$ is such that $j(\alpha)$ is a limit point of $D$, then $D \cap j(\alpha) \in  j(\mathcal C_\alpha)$, and it is also in $V$ since it has size $<\mu$.  If we take $d \in V$ such that $j(d) = D \cap j(\alpha)$, then $d \in \mathcal C_\alpha$.  If $E = j^{-1}[D]$, then $E$ is a club subset of $\gamma$ in $V[G]$ such that for every $\alpha \in \lim E$, $E \cap \alpha \in \mathcal C_\alpha$.

In $V$, let $\la \dot \alpha_i : i < \mu \ra$ name the increasing enumeration of $E$.  Using the strategic closure of $\p(\p_\kappa\lambda')/I$, we build a tree of height $\cf(\lambda)+1$ of conditions that make incompatible decisions about initial segments of $E$.  Suppose inductively that for some $\eta \leq \cf(\lambda)$, we have ordinals $i_\sigma$ and $\xi_\sigma$ and conditions $p_\sigma$, indexed by sequences $\sigma \in \lambda^{<\eta}$, such that:
\begin{enumerate}
\item For each $\sigma \in \lambda^{<\eta}$, $\la p_{\sigma \restriction \gamma} : \gamma \in \dom \sigma \ra$ is a descending sequence conforming to the strategy witnessing the strategic closure of $\p(\p_\kappa\lambda')/I$.
\item If $\dom \sigma < \dom \sigma'$, then $\xi_\sigma < \xi_{\sigma'}$.
\item If $\sigma \not= \sigma'$ and their domains are the same successor ordinal, then $\xi_\sigma \not= \xi_{\sigma'}$.
\item For $\zeta+1 < \eta$, $\sigma \in \lambda^\zeta$, and $\beta < \lambda$, $p_{\sigma ^\frown \beta} \Vdash \dot\alpha_{i_{\sigma}} = \check\xi_{\sigma ^\frown \beta}$, and $p_{\sigma ^\frown \beta}$ decides $E \cap \xi_{\sigma ^\frown \beta}$.
\end{enumerate}
For each $\sigma \in \lambda^\eta$, we can take a lower bound $p_\sigma$ to  $\la p_{\sigma \restriction \gamma} : \gamma < \eta \ra$.  Since $\lambda^+$ is regular, there is $i_\sigma <\mu$ such that the possible values for $\dot\alpha_{i_\sigma}$ forced below $p_\sigma$ are unbounded in $\lambda^+$.  Choose an antichain $\la p_{\sigma ^\frown \beta} : \beta < \lambda \ra$ of conditions below $p_\sigma$ that make distinct decisions $\xi_{\sigma ^\frown \beta}$ for $\dot\alpha_{\sigma}$, all above $\sup_{\tau \in \lambda^{<\eta}} \xi_\tau$, also deciding $E \cap  \xi_{\sigma ^\frown \beta}$, and conforming to the strategy witnessing strategic closure.

Let $\beta^* = \sup_{\sigma \in \lambda^{<\cf(\lambda)}} \xi_\sigma <\lambda^+$.  For each $\sigma \in \lambda^{\cf(\lambda)}$, $p_\sigma$ forces that $\beta^*$ is a limit point of $E$, and so $p_\sigma \Vdash \dot E \cap \beta^* \in \check{\mathcal C}_{\beta^*}$.  But there are $\lambda^{\cf(\lambda)} > \lambda$ distinct decisions, contradicting that $|\mathcal C_{\beta^*}|\leq \lambda$.\end{proof}

\bibliographystyle{amsplain.bst}
\bibliography{masterbib}

\end{document}